\documentclass[11pt]{article}
\usepackage{graphicx}
\usepackage{amsthm, amsmath, amssymb, tikz, bm}
\usepackage{mathtools}
\usepackage{xifthen}

\usepackage[margin=1.2in]{geometry} 

\usepackage[shortlabels]{enumitem}
\usepackage{todonotes}
\usepackage[colorlinks=true,
linkcolor=blue,citecolor=blue,
urlcolor=blue]{hyperref}

\allowdisplaybreaks

\usetikzlibrary{calc,shapes, backgrounds}

\newtheorem{theorem}{Theorem}
\newtheorem{lemma}[theorem]{Lemma}
\newtheorem{corollary}[theorem]{Corollary}

\newtheorem{claim}[theorem]{Claim}
\newtheorem{remark}[theorem]{Remark}
\newtheorem{conjecture}[theorem]{Conjecture}
\newtheorem{proposition}[theorem]{Proposition}

\newcommand{\ds}{\displaystyle}
\newcommand{\dss}{\displaystyle\sum}

\newcommand{\lp}{\left (}
\newcommand{\rp}{\right )}

\newcommand{\cF}{\mathcal{F}}
\newcommand{\cH}{\mathcal{H}}

\DeclarePairedDelimiter{\ceil}{\lceil}{\rceil}

\DeclarePairedDelimiter{\abs}{\lvert}{\rvert}%

\newcommand{\norm}[1]{\left\lVert#1\right\rVert}

\title{On Hamiltonian Berge cycles in $3$-uniform hypergraphs}

\author{
Linyuan Lu
\thanks{University of South Carolina, Columbia, SC 29208,
({\tt lu@math.sc.edu}). This author was supported in part by NSF grant DMS-1600811.} \and
Zhiyu Wang \thanks{University of South Carolina, Columbia, SC 29208,
({\tt zhiyuw@math.sc.edu}).} 
}

\begin{document}

\maketitle

\begin{abstract}
    Given a set $R$, a hypergraph is $R$-uniform if the size of every hyperedge belongs to $R$.
    A hypergraph $\mathcal{H}$ is called \textit{covering} if every vertex pair is contained in some hyperedge in $\mathcal{H}$. In this note, we show that every covering $[3]$-uniform hypergraph on $n\geq 6$ vertices contains a Berge cycle $C_s$ for any $3\leq s\leq n$. As an application, we determine the maximum Lagrangian of $k$-uniform Berge-$C_{t}$-free hypergraphs and Berge-$P_{t}$-free hypergraphs.
\end{abstract}

\section{Introduction}
A \textit{hypergraph} is a pair $\cH=(V,E)$ where $V$ is a vertex set and every \textit{hyperedge} $h \in E$ is a subset of $V$. For a fixed set of positive integers $R$, we say $\cH$ is an \textit{$R$-uniform} hypergraph, or \textit{$R$-graph} for short, if the cardinality of each hyperedge belongs to $R$. 
 If $R=\{k\}$, then an $R$-graph is simply a $k$-uniform hypergraph or a $k$-graph. In this paper, we assume all hypergraphs are \textit{simple}, i.e., $E$ does not have two identical hyperedges and each hyperedge does not contain a vertex more than once. 
Given an $R$-graph $\cH = (V,E)$ and a set $S \in \binom{V}{s}$, let deg$(S)$ denote the number of edges containing $S$ and $\delta_s(\cH)$ be the minimum \textit{$s$-degree} of $\cH$, i.e., the minimum of deg$(S)$ over all $s$-element sets $S \in \binom{V}{s}$. Given a hypergraph $\cH$, the \textit{$2$-shadow} of $\cH$, denoted by $\partial(\cH)$, is a simple $2$-uniform graph $G=(V,E)$ such that $V(G)=V(\cH)$ and $uv\in E(G)$ if and only if $\{u,v\}\subseteq h$ for some $h\in E(\cH)$. In this paper, since we are dealing with $3$-uniform hypergraphs, for convenience we will simply use the term \textit{shadow} instead of $2$-shadow. we say $\cH$ is {\em covering} if the shadow of $\cH$ is a complete graph. Note that $\cH$ is covering if and only if $\delta_2(\cH) \geq 1$.

There are several notions of a path or a cycle in hypergraphs. For $t\geq 2$, a \textit{Berge path} $P$ of length $t-1$, denoted by Berge-$P_t$, is a collection of $t$ vertices $v_1, \ldots, v_{t}$ and $t-1$ distinct hyperedges $h_1, \ldots, h_{t-1}$ such that $\{v_i, v_{i+1}\} \subseteq h_i$ for each $i\in [t-1]$. Similarly, for $t\geq 3$, a \textit{Berge} cycle $C$ of length $t$, denoted by Berge-$C_t$, is a collection of $t$ distinct vertices $v_1, v_2, \ldots, v_t$  and $t$ distinct hyperedges $h_1, h_2, \ldots, h_t$ such that $\{v_i, v_{i+1}\} \subseteq h_i$ for every $i\in [t]$ where $v_{t+1} \equiv v_1$. 
The vertices $v_1, v_2 \cdots, v_t$ are called the \textit{base vertices} of $P$ and $C$. Moreover, we refer $h_i$ as the hyperedge \textit{embedding} $v_i v_{i+1}$.
For convenience, we also use $P = v_1 v_2 \cdots v_{t}$ and $C = v_1 v_2 \cdots v_t$ to denote a Berge path of length $t-1$ and a Berge cycle of length $t$ with base vertices $\{v_1, \cdots,v_t\}$. We say an $R$-graph $\cH$ on $n$ vertices contains a \textit{Hamiltonian} Berge cycle (path) if it contains a Berge cycle (path) of length $n$ (or $n-1$). 
More generally, Gerbner and Palmer \cite{Gerbner-Palmer17} extended the definition of Berge paths and Berge cycles to general graphs. In particular, given  a simple graph $G$, a hypergraph $\cH$ is called a \emph{Berge-$G$} hypergraph if there is
an injection $i\colon V(G)\to V(\mathcal{H})$ and 
a bijection $f\colon E(G) \to E(\mathcal{H})$ such that for all $e=uv \in E(G)$, we have $\{i(u), i(v)\} \subseteq f(e)$.

For $k$-uniform hypergraphs, there are more structured notions of Berge cycles as well. Given  $1\leq \ell < k$, a $k$-graph $C$ is called an $\ell$-cycle if its vertices can be ordered cyclically such that each of its edges consists of $k$ consecutive vertices and every two consecutive edges (in the natural order of the edges) share exactly $\ell$ vertices. In particular, in a $k$-graph, a $(k-1)$-cycle is often called a \textit{tight cycle} while a $1$-cycle is often called a \textit{loose cycle}. A $k$-graph contains a \textit{Hamiltonian $\ell$-cycle} if it contains an $\ell$-cycle as a spanning subhypergraph. 

The problem of finding Hamiltonian cycles has been widely studied. In 1952, Dirac \cite{Dirac52} showed that for $n\geq 3$, every $n$-vertex graph with minimum degree  at least $n/2$ contains a Hamiltonian cycle. Since then, problems that relate the minimum degree (or minimum $s$-degree in hypergraphs) to the structure of the (hyper)graphs are often referred to as \textit{Dirac-type problems}. In the setting of hypergraphs, define the threshold $h^{\ell}_s(k,n)$ as the smallest integer $m$ such that every $k$-graph $\cH$ on $n$ vertices with $\delta_s(\cH) \geq m$ contains a Hamiltonian $\ell$-cycle, provided that $k-\ell$ divides $n$. These thresholds for different values of $s$, $\ell$ and $k$ have been intensively studied in a series of papers (e.g., \cite{Katona-Kierstead99, RRE06, RRE08, RRE11, Treglown-Zhao12, Markstrom-Rucinski11, KMO10, Han-Zhao15}, see \cite{Zhao16} for a recent survey). For Berge cycles, Bermond, Germa, Heydemann, and Sotteau \cite{BGHS78} showed a Dirac-type theorem for Berge cycles. Kostochka, Luo and Zirlin \cite{KLZ19} recently showed some Dirac-type conditions for a hypergraph with few edges to be Hamiltonian.

The problem of finding Hamiltonian Berge cycles in a hypergraph is closely related to the problem of finding rainbow Hamiltonian cycles in an edge-colored complete graph $K_n$. 
An edge-colored graph $G$ is \textit{rainbow} (or \textit{multicolored}) if each edge is of a different color. An edge-colored graph $G$ is \textit{$k$-bounded} if no color appears in more than $k$ edges.
Observe that given any covering $k$-graph $\cH$ with hyperedges $h_1, \cdots, h_m$, we can construct an edge-colored complete graph $G$ (using colors $\{c_1, \cdots, c_m\}$) on $|V(\cH)|$ vertices by assigning any edge $uv \in E(G)$ color $c_i$ if $uv \in h_i$ for some $i$ (pick arbitrarily if $uv$ is contained in multiple hyperedges). Notice that $G$ is $\binom{k}{2}$-bounded. Moreover, any rainbow subgraph $G'$ of $G$ corresponds to a Berge-$G'$ in $\cH$ by embedding $uv \in E(G')$ into the hyperedge $h_i$ if $uv$ is colored $c_i$.

There have been intensive investigations on the largest $k$ (compared to $n$) such that any $k$-bounded edge-coloring of $K_n$ contains a rainbow Hamiltonian path or cycle. In this framework, Hahn \cite{Hahn80} conjectured that any $(n/2)$-bounded coloring of $K_n$ contains a rainbow Hamiltonian path. Hahn's conjecture was disproved by Maamoun and Meyniel \cite{Maamoun-Meyniel84} who showed that the conjecture is not true for proper colorings of $K_{2^t}$ for integers $t \geq 2$.
The problem for rainbow Hamilton cycles was first mentioned in Erd\H{o}s, Nesdtril and R\"{o}dl \cite{ENR83} as an Erd\H{o}s-Stein problem and show that $k$ can be any constant. 
Hahn and Thomassen \cite{Hahn-Thomassen86} showed that $k$ could grow as fast as $n^{1/3}$ and conjectured that the growth rate of $k$ can be linear. R\"{o}dl and Winkler later in an unpublished work improved it to $n^{1/2}$. Frieze and Reed \cite{Frieze-Reed93} improved it to $O(n/\ln n)$. Albert, Frieze and Reed \cite{AFR95} confirmed the conjecture of Hahn and Thomassen by showing that if $n$ is sufficiently large and $k$ is at most $\ceil{cn}$ where $c<\frac{1}{32}$, then any $k$-bounded edge-coloring of $K_n$ contains a rainbow Hamiltonian cycle. Frieze and Krivelevich \cite{Frieze-Krivelevich08} showed that there exists absolute constant $c > 0$ such that if an edge-coloring of $K_n$ is $cn$-bounded, then there exists rainbow cycles of all sizes $3 \leq \ell \leq n$. In the context of Berge Hamiltonian cycles, the results above imply the following theorem:

\begin{theorem}\cite{ENR83, Hahn-Thomassen86, Frieze-Reed93, AFR95}\label{thm:cover-cycles}
For any fixed set of integers $R\subseteq [k]$ where $k \geq 2$, there is an integer $n_0:=n_0(k)$ such that every covering $R$-graph $\cH$ on at least $n_0$ vertices contains Berge cycles of all sizes $3\leq \ell \leq n$. 
\end{theorem}

\begin{corollary}\label{cor:cover-path}
For any fixed set of integers $R\subseteq [k]$ where $k \geq 2$, there is an integer $n_0:=n_0(k)$ such that every covering $R$-graph $\cH$ on at least $n_0$ vertices contains a Berge Hamiltonian path.
\end{corollary}

Further results on the rainbow spanning subgraphs lead to results that are even stronger than Theorem \ref{thm:cover-cycles}. In particular, B\"{o}ttcher, Kohayakawa and Procacci \cite{BKP12} showed that for $c\leq n/(51\Delta^2)$ every $cn$-bounded $K_n$ contains a rainbow copy of every graph with maximum degree $\Delta$. Recently, Coulson and Perarnau \cite{Coulson-Perarnau18} showed that there exists $c >0$ such that if $G$ is a Dirac graph (i.e. minimum degree at least $n/2$) on $n$ vertices (for sufficiently large $n$), then any $cn$-bounded coloring of $G$ contains a rainbow Hamiltonian cycle. 

The results above assume $n$ is sufficiently large with respect to $k$. In this paper, we prove more precise results and focus on the Hamiltonian Berge paths and cycle problems in $[3]$-uniform hypergraphs (i.e., all hyperedges have cardinality at most $3$). In particular, we show the following theorems: 

\begin{theorem}\label{thm:3-path}
Every covering $[3]$-graph $\cH$ on $n \geq 4$ vertices contains a Hamiltonian Berge path.
\end{theorem}

\begin{theorem}\label{thm:3-cycle}
Every covering $[3]$-graph $\cH$ on $n \geq 6$ contains a Berge cycle $C_s$ for any $3\leq s \leq n$.
\end{theorem}

Note that in Theorem \ref{thm:3-cycle} the lower bound on $n$ is best possible since the hypergraph containing only the $3$-uniform hyperedges shown in Figure \ref{fig:counters}(b) is a covering $[3]$-graph on $5$ vertices with only $4$ hyperedges, thus without a Hamiltonian Berge cycle. 


\begin{figure}[htb]
\begin{center}
    \begin{minipage}{0.3\textwidth}
        \resizebox{2.5cm}{!}{\begin{tikzpicture}
        \tikzstyle{vertex}=[circle,fill=black,inner sep=1pt]

        \node[style=vertex,label=above:{$v_{1}$}] (v1) at (1, 0) {};
        \node[style=vertex,label=above:{$v_{2}$}] (v2) at (-1, 0) {};
        \node[style=vertex,label=below:{$v_{3}$}] (v3) at (0, -1) {};
        \node[style=vertex,label=above:{$w$}] (w) at (0, 1) {};

        \fill[red, opacity=0.4] (0,1)--(1,0)--(-1,0)--cycle;

        \draw (v2) -- (v3);
        \draw (v3) -- (v1);
        \draw (v3) -- (w);
        
\end{tikzpicture}}
    \end{minipage} \hspace{0.1cm}
    \begin{minipage}{0.3\textwidth}
        \resizebox{2.5cm}{!}{\begin{tikzpicture}
        \tikzstyle{vertex}=[circle,fill=black,inner sep=1pt]

        \node[style=vertex,label=above:{$v_{1}$}] (v1) at (1, 1) {};
        \node[style=vertex,label=above:{$v_2$}] (v2) at (0, 1) {};
        \node[style=vertex,label=above:{$v_{3}$}] (v3) at (0, 0) {};
        \node[style=vertex,label=below:{$v_{4}$}] (v4) at (1, 0) {};
        \node[style=vertex,label=above:{$v_{5}$}] (v5) at (1.5, 0.5) {};

        \fill[red, opacity=0.4] (1, 1)--(0, 1)--(1.5, 0.5)--cycle;
        \fill[blue, opacity=0.4] (0, 0)--(1, 0)--(1.5, 0.5)--cycle;
        \fill[green, opacity=0.4] (1, 1)--(0, 1)--(0, 0)--cycle;
        \fill[yellow, opacity=0.4] (1, 1)--(0, 1)--(1, 0)--cycle;
        
        \draw (v1) -- (v2);

\end{tikzpicture}}
    \end{minipage} 
    \caption{Two hypergraphs on $n$ vertices with $n$ hyperedges without Hamiltonian cycles.}
    \label{fig:counters}

\end{center}
\end{figure}
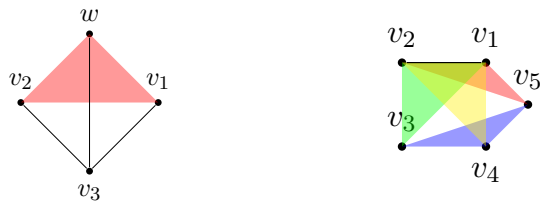

\begin{remark}\label{rmk:small-counter}
In fact, except the two hypergraphs in Figure \ref{fig:counters}, all covering $\{2,3\}$-graphs on $n$ vertices with at least $m\geq 3$ hyperedges contain a Berge cycle of length $s$ where $3 \leq s \leq \min(m,n)$. Observe that both hypergraphs in Figure \ref{fig:counters} contain $2$-uniform hyperedges. 
\end{remark}

Based on the observation in Remark \ref{rmk:small-counter}, we conjecture the following:
\begin{conjecture}
For all $k\geq 4$, every $k$-uniform covering hypergraph on $n$ vertices with at least $m\geq 3$ hyperedges contains a Berge cycle of length $s$ for $3 \leq s\leq \min(m,n)$.
\end{conjecture}

As an application, using Theorem \ref{thm:cover-cycles}, \ref{thm:3-path} and \ref{thm:3-cycle}, we determine the maximum \textit{Lagrangian} of Berge-$P_t$-free and Berge-$C_{t}$-free $k$-graphs when $t$ is sufficiently large. 
Given a $k$-uniform hypergraph $\cH$ on $n$ vertices, the polynomial form $P_{\cH}(\bm{x}):\mathbb{R}^n \to \mathbb{R}$ is defined for any vector $\bm{x}= (x_1, \ldots, x_n) \in \mathbb{R}^n$ as 
$$P_{\cH}(\bm{x}) = \dss_{\{i_1, i_2, \cdots, i_k\} \in E(\cH)} x_{i_1} \cdots x_{i_k}.$$
For $k\geq 2$,
the \textit{Lagrangian} of a $k$-uniform hypergraph $\cH = (V,E)$ on $n$ vertices is defined to be 
$$\lambda(\cH) = \ds\max_{\bm{x}\in \mathbb{R}_{\ge 0}^n: \norm{\bm{x}}_1 = 1}  P_{\cH}(\bm{x}).$$
where the $\norm{\bm{x}}_1 =  \dss_{i=1}^n \abs{x_i}$ is the \textit{1-norm} of $\bm{x} \in \mathbb{R}^n$.
Lagrangians for graphs (i.e., $2$-graphs) were introduced by Motzkin and Straus in 1965 \cite{Motzkin-Straus65}. They showed $\lambda(G)=\frac{1}{2}(1-\frac{1}{\omega(G)})$,
where $\omega(G)$ is the clique number of $G$.
The Lagrangian of a $k$-graph $\cH$ is closely related to the maximum edge density of the blow-up of $\cH$, which is frequently used in extremal hypergraph theory \cite{Talbot02, Keevash11}.

There have been intensive studies on the hypergraph Tur\'an numbers of Berge paths and cycles, which are concerned with the maximum number of hyperedges in a $k$-uniform hypergraph without a Berge path or cycle of certain length (see for example \cite{GKL10, DGMT18, Gyori-Lemons12, Gyori-Lemons12b, Kostochka-Luo18, FKL19, FKL19B, EGMSTZ19, GLSZ18}). Here we consider the maximum Lagrangian of Berge-$C_t$-free and Berge-$P_{t}$-free hypergraphs respectively.

\begin{theorem}\label{thm: max-Lagrangian-cycle}
For fixed $k \geq2 $ and sufficiently large $t=t(k)$ and $n \geq t-1$, let $\cH$ be a $k$-uniform hypergraph on $n$ vertices without a Berge cycle of length $t$. Then 
$$\lambda(\cH) \leq \lambda(K^k_{t-1}) = \frac{1}{(t-1)^{k}} \binom{t-1}{k}.$$
\end{theorem}

As a corollary, we obtain the same results for the Berge-$P_t$-free hypergraphs as well.

\begin{corollary}\label{cor: max-Lagrangian-path}
For fixed $k \geq2 $ and sufficiently large $t=t(k)$ and $n \geq t-1$, let $\cH$ be a $k$-uniform hypergraph on $n$ vertices without a Berge-$P_t$. Then 
$$\lambda(\cH) \leq \lambda(K^k_{t-1}) = \frac{1}{(t-1)^{k}} \binom{t-1}{k}.$$
\end{corollary}

Both the bounds in Theorem \ref{thm: max-Lagrangian-cycle} and Corollary \ref{cor: max-Lagrangian-path} are tight. Indeed, let $\cH$ be a $k$-graph obtained from $K^k_{t-1}$ by adding $(n-t+1)$ isolated vertices. Clearly $\cH$ is Berge-$C_t$-free and Berge-$P_t$-free and $\lambda(\cH) = \binom{t-1}{k}/(t-1)^k$. For $k = 3$, due to Theorem \ref{thm:3-path} and Theorem \ref{thm:3-cycle}, we obtain more precise results.
\begin{corollary}\label{cor:3-cycle-Lagrangian}
Let $\cH$ be a $3$-uniform hypergraph on $n$ vertices without a Berge-$C_t$ where $n\geq t\geq 6$. Then 
$$\lambda(\cH) \leq\lambda(K^3_{t-1}) = \frac{1}{(t-1)^{3}} \binom{t-1}{3}.$$
\end{corollary}

\begin{corollary}\label{cor:3-path-Lagrangian}
Let $\cH$ be a $3$-uniform hypergraph on $n$ vertices without a Berge-$P_t$ where $n\geq t\geq 4$. Then 
$$\lambda(\cH) \leq\lambda(K^3_{t-1}) = \frac{1}{(t-1)^{3}} \binom{t-1}{3}.$$
\end{corollary}

\section{Proof of Theorem \ref{thm:3-path} }

\begin{proof}[Proof of Theorem \ref{thm:3-path}]
Let $\cH = (V,E)$ be a covering $[3]$-uniform hypergraph on $n\geq 4$ vertices. Let $P = v_1 v_2 \ldots v_t$ be a maximum-length Berge path in $\cH$. If $t  = n$, we are done. Otherwise assume that $t < n$ and let $u$ be a vertex that is not a base vertex of $P$. Observe that by the maximality of $P$, we have $t\geq 3$. Call a hyperedge $h$ \textit{used} if $h$ is an edge in the Berge path $P$, otherwise call it \textit{free}. Since $\cH$ is covering, there exists a hyperedge $h_1$ containing $\{u,v_1\}$. The edge $h_1$ must be used in $P$ since otherwise we can extend $P$ by embedding $\{u, v_1\}$ in $h_1$. Since $\cH$ is $[3]$-uniform, the only way that $h_1$ can be used in $P$ is to embed $\{v_1, v_2\}$. Similarly, there exists a hyperedge $h_t$ that contains $\{u, v_t\}$ and is used to embed $\{v_{t-1},v_t\}$. Now consider a hyperedge $h'$ containing $\{v_1, v_t\}$. Note that $h'$ is free since both $\{v_1, v_2\}$ and $\{v_{t-1},v_t\}$ have already been embedded. Now consider the path 
$$P' = v_2 v_3 \cdots v_{t-1} v_t v_1 u$$
such that $\{v_t, v_1\}$ is embedded in $h'$, $\{v_1, u\}$ is embedded in $h_1$ and any other $2$-edge in $P'$ is embedded the same way as in $P$. Notice that $P'$ is a Berge hyperpath in $\cH$ that is longer than $P$. This gives us the contradiction. Hence $t = n$ and $P$ is a Hamiltonian Berge path in $\cH$.
\end{proof}

\section{Proof of Theorem \ref{thm:3-cycle}}

\begin{lemma}\label{lem:3-Hamiltonian-cycle}
Let $\cH = (V,E)$ be a covering $[3]$-graph on $n\geq 6$ vertices. Then $\cH$ contains a Hamiltonian Berge cycle.
\end{lemma}

\begin{proof}[Proof of Lemma \ref{lem:3-Hamiltonian-cycle}]
Let $\cH = (V,E)$ be a covering $[3]$-graph on $n\geq 6$ vertices. Suppose for the sake of contradiction that $\cH$ does not contain a Hamiltonian Berge cycle.

We first claim that there then exists a Berge cycle of length $n-1$. By Theorem \ref{thm:3-path}, there is a Hamiltonian Berge path $P = u_1 u_2 \ldots u_{n}$ in $\cH$. Since $\cH$ is covering, if follows that there exists an edge $h \in E(\cH)$ such that $\{u_1, u_{n}\} \subseteq h$. If $h$ is not an edge in $P$, then we embed $u_1 u_{n}$ in $h$ and obtain a Hamiltonian Berge cycle. Otherwise, $h$ is used to embed either $u_1u_2$ or $u_{n-1}u_{n}$. WLOG, $h$ embeds $u_{n-1} u_{n}$. Then $h = \{u_1, u_{n-1}, u_{n}\}$. If we embed $u_1 u_{n-1}$ in $h$, we then obtain a Berge cycle $C = u_1 u_2 \ldots u_{n-1}$ of length $n-1$.

Let $C = v_1 v_2 \ldots v_{n-1}$ be a Berge cycle in $\cH$ of length $n-1$ and call the remaining vertex $w$. For ease of reference, consider $v_{n} \equiv v_1$ and $v_0 \equiv v_{n-1}$. For a $2$-edge $e = v_i v_{i+1}$ in $C$, we use $\phi(e)$ to denote the hyperedge in $C$ that embeds $e$. 
Consider a two-edge-coloring on $\{v_i v_{i+1}: i\in [n-1]\}$: color $v_iv_{i+1}$ red if the hyperedge that embeds $v_iv_{i+1}$ also contains $w$; otherwise color it blue. Assume that $C$ is picked among all Berge cycles of length $n-1$ such that $C$ has the most number of red edges (when viewed as a $2$-uniform cycle). 

Again, from now on, we call a hyperedge $h$ \textit{used} if $h$ is a hyperedge in $C$, otherwise call it \textit{free}. Moreover, when we say 2-edges of $C$, we mean the 2-uniform edges of $C$ when $C = v_1 v_2 \ldots v_{n-1}$ is viewed as a 2-uniform cycle. Otherwise, $C$ is considered a $[3]$-graph.

\begin{figure}[htb]
\begin{center}
    \begin{minipage}{0.3\textwidth}
        \resizebox{4cm}{!}{ \begin{tikzpicture}
        \tikzstyle{vertex}=[circle,fill=black,inner sep=1pt]
        
        \pgfmathsetmacro\unitangle{360/9}
        \pgfmathsetmacro\offset{1}

        \node[style=vertex,label=right:{$v_1$}] (v1) at (9 * \unitangle: \offset) {};
        \node[style=vertex,label=above:{$v_2$}] (v2) at (1 * \unitangle: \offset) {};
        \node[style=vertex,label=above:{$v_i$}] (v3) at (2 * \unitangle: \offset) {};
        \node[style=vertex,label=above:{$v_{i+1}$}] (v4) at (3 * \unitangle: \offset) {};
        \node[style=vertex,label=left:{$v_{i+2}$}] (v5) at (4 * \unitangle: \offset) {};
        \node[style=vertex] (v6) at (5 * \unitangle: \offset) {};
        \node[style=vertex,label=below:{$v_{j}$}] (v7) at (6 * \unitangle: \offset) {};
        \node[style=vertex,label=below:{$v_{j+1}$}] (v8) at (7 * \unitangle: \offset) {};
        \node[style=vertex,label=right:{$v_{n-1}$}] (v9) at (8 * \unitangle: \offset) {};

        \draw[blue] (v1) -- (v2);
        \draw[blue, dashed] (v2) -- (v3);
        \draw[red] (v3) -- (v4);
        \draw[blue] (v4) -- (v5);
        \draw[blue, dashed] (v5) -- (v6);
        \draw[blue] (v6) -- (v7);
        \draw[red] (v7) -- (v8);
        \draw[blue, dashed] (v8) -- (v9);
        \draw[blue] (v9) -- (v1);

        \draw (v4) -- (v8);
\end{tikzpicture}}
    \end{minipage} \hspace{0.5cm}
    \begin{minipage}{0.3\textwidth}
        \resizebox{4cm}{!}{ \begin{tikzpicture}
        \tikzstyle{vertex}=[circle,fill=black,inner sep=1pt]
        
        \pgfmathsetmacro\unitangle{360/9}
        \pgfmathsetmacro\offset{1}

        \node[style=vertex,label=right:{$v_1$}] (v1) at (9 * \unitangle: \offset) {};
        \node[style=vertex,label=above:{$v_2$}] (v2) at (1 * \unitangle: \offset) {};
        \node[style=vertex,label=above:{$v_i$}] (v3) at (2 * \unitangle: \offset) {};
        \node[style=vertex,label=above:{$v_{i+1}$}] (v4) at (3 * \unitangle: \offset) {};
        \node[style=vertex,label=left:{$v_{i+2}$}] (v5) at (4 * \unitangle: \offset) {};
        \node[style=vertex] (v6) at (5 * \unitangle: \offset) {};
        \node[style=vertex,label=below:{$v_{j}$}] (v7) at (6 * \unitangle: \offset) {};
        \node[style=vertex,label=below:{$v_{j+1}$}] (v8) at (7 * \unitangle: \offset) {};
        \node[style=vertex,label=right:{$v_{n-1}$}] (v9) at (8 * \unitangle: \offset) {};
        \node[style=vertex,label=right:{$w$}] (w) at (0.2,0) {};

        \draw[blue] (v1) -- (v2);
        \draw[blue, dashed] (v2) -- (v3);
        \draw[blue] (v4) -- (v5);
        \draw[blue, dashed] (v5) -- (v6);
        \draw[blue] (v6) -- (v7);
        \draw[blue, dashed] (v8) -- (v9);
        \draw[blue] (v9) -- (v1);

        \draw (w) -- (v3);
        \draw (w) -- (v7);
        \draw (v4) -- (v8);

\end{tikzpicture}}
    \end{minipage} 
    \caption{Using a bridge to extend the cycle.}
    \label{fig:bridge}

\end{center}
\end{figure}

\begin{claim}\label{cl:two-red-one-free}
If there exist two disjoint red pairs $v_iv_{i+1}, v_jv_{j+1}$ such that there is a free edge $h$ containing either $v_i v_j$ or $v_{i+1} v_{j+1}$, then we have a Hamiltonian Berge cycle. 
\end{claim}
\begin{proof}
Recall that $\phi(v_k v_{k+1})$ denotes the hyperedge in $C$ that embeds $v_k v_{k+1}$. Suppose there is a free edge $h$ containing $v_{i+1}v_{j+1}$ (as shown in Figure \ref{fig:bridge}).
Consider the cycle $$C' = v_i w v_j v_{j-1} \ldots v_{i+1} v_{j+1} v_{j+2} \ldots v_i.$$ Embed $v_i w$ in $\phi(v_i v_{i+1})$; embed $w v_j$ in $\phi(v_j v_{j+1})$; embed $v_{i+1}v_{j+1}$ in $h$. For any other edge $e$ of $C'$, embed $e$ in $\phi(e)$. It's easy to see that $C'$ is a Hamiltonian Berge cycle.
\end{proof}

Observe that given two disjoint red pairs $v_iv_{i+1}$, $v_j v_{j+1}$, if the hyperedge $h$ containing $v_iv_j$ is not free, then it must be used to embed either $v_{i-1}v_i$ or $v_{j}v_{j-1}$. Similarly, if the hyperedge containing $v_{i+1}v_{j+1}$ is not free, then it must be used to embed either $v_{i+1}v_{i+2}$ or $v_{j+1}v_{j+2}$. For convenience, given vertex-disjoint red pairs $v_iv_{i+1}$, $v_j v_{j+1}$, we call the vertex pairs $v_{i}v_{j}$ and $v_{i+1}v_{j+1}$ \textit{bridges}. By Claim \ref{cl:two-red-one-free}, if a bridge is free, then we are done. Otherwise by the above observation, a bridge must be used to embed a blue $2$-edge in $C$ that intersects the bridge. 
Call a sequence of vertices a \textit{segment} if they are consecutive in $C$. A segment is red (or blue) if the $2$-edges in $C$ (viewed as a $2$-uniform cycle) induced by the vertices in the segment are all red (or blue).
By Claim \ref{cl:two-red-one-free}, it is easy to derive the following consequences:
\begin{enumerate}[(C1)]
    \item There are no four pairwise disjoint red segments. This is because, for any four pairwise disjoint red segments, there are at least $2 \binom{4}{2} = 12$ bridges but only at most $8$ blue edges that intersects the four red segments. Hence one of the bridges must be free. Then we are done by Claim \ref{cl:two-red-one-free}.
    
    \item If there are three pairwise disjoint red segments,  there must be at least two blue edges (in both directions) between every two red segments. Moreover, each of the red segments has length $1$. This is because, three pairwise disjoint red segments have at least six bridges. If there is only one blue edge between two of the red segments, then there are at most five blue edges intersecting the red segments. Hence one of the bridges must be free and we are done by Claim \ref{cl:two-red-one-free}. Similarly, if there exists one red segment that has length at least $2$, then there will be at least $10$ bridges. Since there are at most $6$ blue edges intersecting these three red segments, it follows that at least one of the bridges is free and we are done. 
    
    \item There can be only one red segment of length at least $2$. Moreover, if there is any other red segment, then there must be at least two blue edges (in both directions) between the two red segments. The logic is similar: if there are two red segments of length at least $2$, then there are at least $7$ bridges but only at most $4$ blue edges intersecting these two red segments. Hence we have a free bridge and we are done. Now suppose there is one red segment of length at least $2$ and one red segment of length $1$, then there are at least $4$ bridges. Hence there must be at least two blue edges (in both directions) between the two red segments.

    \item If there is a red segment of length $3$, there is no other red segment. Otherwise, there are at least $6$ bridges but at most $4$ blue edges intersecting these two red segments. Hence we have a free bridge and we are done.
    
    \item There is no red segment of length at least $4$. Otherwise there are at least $5$ bridges but at most $2$ blue edges intersecting this red segment. Hence we have a free bridge and we are done.
\end{enumerate}

\begin{claim}\label{cl:3-white}
If there exist three consecutive blue edges in $C$, i.e., $v_i, v_{i+1}, v_{i+2}, v_{i+3}$ such that $v_{k}v_{k+1}$ is blue for $k \in \{i,i+1,i+2\}$, then we have a Hamiltonian Berge cycle.
\end{claim}
\begin{proof}
Since $\cH$ is covering, it follows that there exist free edges $h_1, h_2$ such that $h_1$ contains $w v_{i+1}$ and $h_2$ contains $w v_{i+2}$. Note that $h_1 \neq h_2$ otherwise we have a free $h = \{w, v_{i+1}, v_{i+2}\}$, which contradicts our assumption that $C$ is picked such that it has the maximum number of red edges. Now consider the cycle 
$$C' = v_1 \ldots v_{i+1} w v_{i+2} \ldots v_{n-1}.$$
Embed $v_{i+1}w$ in $h_1$; embed $wv_{i+2}$ in $h_2$; embed any other edge $e$ the same way it is embedded in $C$. We then obtain a Hamiltonian Berge cycle.
\end{proof}

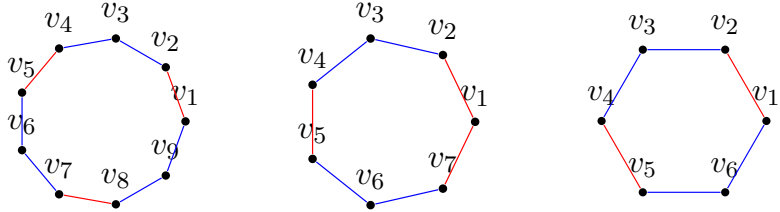
\begin{figure}[htb]
\begin{center}
    \begin{minipage}{.2\textwidth}
        \resizebox{3cm}{!}{ \begin{tikzpicture}
        \tikzstyle{vertex}=[circle,fill=black,inner sep=1pt]
        \pgfmathsetmacro\unitangle{360/9}
        \pgfmathsetmacro\offset{1}
        \foreach \j in {0,...,8}{
            \pgfmathsetmacro\index{int(\j+1)}
        	\node[style=vertex,label=above:{$v_{\index}$}] (y{\j,0}) at (\j * \unitangle: \offset) {};
        }
        \foreach \j in {0,...,8}{
        	\pgfmathsetmacro\jump{int(Mod(int(\j+1),9))};
        	\ifthenelse{\j = 0 \OR \j=3 \OR \j=6}
        	{\draw[red] (y{\j,0}) -- (y{\jump,0});}
        	{\draw[blue] (y{\j,0}) -- (y{\jump,0});}
        }
\end{tikzpicture}}
    \end{minipage} \hspace{0.5cm}
    \begin{minipage}{.2\textwidth}
        \resizebox{3cm}{!}{\begin{tikzpicture}
        \tikzstyle{vertex}=[circle,fill=black,inner sep=1pt]
        \pgfmathsetmacro\unitangle{360/7}
        \pgfmathsetmacro\offset{1}
        
        \foreach \j in {0,...,6}{
            \pgfmathsetmacro\index{int(\j+1)}
        	\node[style=vertex,label=above:{$v_{\index}$}] (y{\j,0}) at (\j * \unitangle: \offset) {};
        }
        \foreach \j in {0,...,6}{
        	\pgfmathsetmacro\jump{int(Mod(int(\j+1),7))};
        	\ifthenelse{\j = 6 \OR \j=0 \OR \j=3}
        	{\draw[red] (y{\j,0}) -- (y{\jump,0});}
        	{\draw[blue] (y{\j,0}) -- (y{\jump,0});}
        }
        \end{tikzpicture}}
    \end{minipage} \hspace{0.5cm}
    \begin{minipage}{.2\textwidth}
        \resizebox{3cm}{!}{        
\begin{tikzpicture}
        \tikzstyle{vertex}=[circle,fill=black,inner sep=1pt]
        \pgfmathsetmacro\unitangle{360/6}
        \pgfmathsetmacro\offset{1}
        \foreach \j in {0,...,5}{
            \pgfmathsetmacro\index{int(\j+1)}
        	\node[style=vertex,label=above:{$v_{\index}$}] (y{\j,0}) at (\j * \unitangle: \offset) {};
        }
        \foreach \j in {0,...,5}{
        	\pgfmathsetmacro\jump{int(Mod(int(\j+1),6))};
        	\ifthenelse{\j = 0 \OR \j=3}
        	{\draw[red] (y{\j,0}) -- (y{\jump,0});}
        	{\draw[blue] (y{\j,0}) -- (y{\jump,0});}
        }
\end{tikzpicture}}
    \end{minipage}
    
    \caption{Remaining cases: for $n\geq 7$. (a): $n=10$; (b): $n=8$; (c): $n=7$.}
    \label{fig:remaining}

\end{center}
\end{figure}

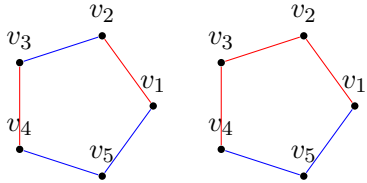
\begin{figure}[htb]
\begin{center}
    \begin{minipage}{.15\textwidth}
        \resizebox{2.5cm}{!}{\begin{tikzpicture}
        \tikzstyle{vertex}=[circle,fill=black,inner sep=1pt]
        \pgfmathsetmacro\unitangle{360/5}
        \pgfmathsetmacro\offset{1}
        
        \foreach \j in {0,...,4}{
            \pgfmathsetmacro\index{int(\j+1)}
        	\node[style=vertex,label=above:{$v_{\index}$}] (y{\j,0}) at (\j * \unitangle: \offset) {};
        }
        \foreach \j in {0,...,4}{
        	\pgfmathsetmacro\jump{int(Mod(int(\j+1),5))};
        	\ifthenelse{\j=0 \OR \j=2}
        	{\draw[red] (y{\j,0}) -- (y{\jump,0});}
        	{\draw[blue] (y{\j,0}) -- (y{\jump,0});}
        }
        \end{tikzpicture}}
    \end{minipage} \hspace{0.1cm}
    \begin{minipage}{.15\textwidth}
        \resizebox{2.5cm}{!}{\begin{tikzpicture}
        \tikzstyle{vertex}=[circle,fill=black,inner sep=1pt]
        \pgfmathsetmacro\unitangle{360/5}
        \pgfmathsetmacro\offset{1}
        
        \foreach \j in {0,...,4}{
            \pgfmathsetmacro\index{int(\j+1)}
        	\node[style=vertex,label=above:{$v_{\index}$}] (y{\j,0}) at (\j * \unitangle: \offset) {};
        }
        \foreach \j in {0,...,4}{
        	\pgfmathsetmacro\jump{int(Mod(int(\j+1),5))};
        	\ifthenelse{\j = 1 \OR \j=0 \OR \j=2}
        	{\draw[red] (y{\j,0}) -- (y{\jump,0});}
        	{\draw[blue] (y{\j,0}) -- (y{\jump,0});}
        }
        \end{tikzpicture}}
    \end{minipage} \hspace{0.1cm}
    \caption{Remaining cases for $n = 6$. (a): $2$ red segments; (b): $1$ red segment.}
    \label{fig:remaining2}
\end{center}
\end{figure}

\begin{claim}\label{cl:finite}
If $\cH$ contains no Hamiltonian cycle, then there are only $5$ possible red-blue colorings of $C$, all listed in Figure \ref{fig:remaining} and \ref{fig:remaining2}.
\end{claim}
\begin{proof}

Suppose there are $s$ red segments in $C$ of decreasing length $\vec{r} = \{r_1, r_2, \cdots, r_s\}$. By consequence (C1) above, $s\leq 3$. By consequence (C5), $r_1 \leq 3$. 

\begin{description}
    \item Case (A): $r_1 = 3$. In this case, by consequence (C4), all other edges of $C$ are blue. Now by Claim \ref{cl:3-white}, there can be at most $2$ other blue edges. Consequence (C4) again implies there has to be at least two blue edges. Hence the only possible such cycle is the one in Figure  \ref{fig:remaining2}(b).
    
    \item Case (B): $r_2 = 2$. In this case, by consequence (C1)-(C3), $\vec{r}$ can either be $\{2,1\}$ or $\{2\}$. If $\vec{r} = \{2\}$, we have a blue segment of length $4$ since $n\geq 6$. By Claim \ref{cl:3-white}, $\cH$ contains a Hamiltonian Berge cycle, which is a contradiction. Hence $\vec{r} = \{2,1\}$. Now (C3) implies there are at least two blue edges between the two red segments in each direction. Since there is no blue segment of length $3$, it follows that Figure \ref{fig:remaining}(b) is the only possible option.

    \item Case (C): $r_1 = 1$. Then $\vec{r}$ can be $\{1,1,1\}, \{1,1\}$, or $\{1\}$. If $\vec{r} = \{1,1,1\}$, then (C2) and Claim \ref{cl:3-white} implies that Figure \ref{fig:remaining}(a) is the only option; if $\vec{r} = \{1,1\}$, it easily follows from Claim \ref{cl:3-white} and $n\geq 6$ that Figure \ref{fig:remaining}(c), Figure \ref{fig:remaining2}(a) are the only options; the last case is not possible since we assume $n\geq 6$.
    
\end{description}
    This completes the proof of the claim.
\end{proof}

Now we show that in each of the cases above, we can extend $C$ into a Hamiltonian Berge cycle, which gives us a contradiction.

\begin{description}
    \item Case 1: $n= 10$. By Claim \ref{cl:finite}, it suffices to consider the cycle in Figure \ref{fig:remaining}(a). In this case, observe there must be a free hyperedge containing each of $wv_3$, $w v_6$ and $w v_9$. Moreover, the free hyperedges containing $wv_3$, $w v_6$ and $w v_9$ cannot be the same hyperedge. Hence, WLOG, let $h_1$ be the free edge containing $w v_3$ and $h_2$ be the free hyperedge containing $w v_9$. Now observe that $v_2 v_8$ is bridge. Let $h$ be an hyperedge containing $v_2 v_8$. If $h$ is free, we are done by Claim \ref{cl:two-red-one-free}. Otherwise, WLOG, $h$ is used to embed $v_2 v_3$ or $v_8 v_9$. In either case, consider the cycle
    $$C' = v_2 v_8 v_7 \ldots v_3 w v_9 v_1v_2$$
    where $v_2 v_8$ is embedded in $h$; $v_3w$ is embedded in $h_1$; $w v_9$ is embedded in $h_2$; and any other $2$-edge of $C'$ is embedded the same way as in $C$.

    \item Case 2:  $n= 8$. By Claim \ref{cl:finite}, it suffices to consider the cycle in Figure \ref{fig:remaining}(b). Note that $v_4v_1$ is a bridge. Hence if the edge $h$ containing $v_4 v_1$ is free, then we are done by Claim \ref{cl:two-red-one-free}. Otherwise, $h$ must be used to embed $v_3 v_4$, i.e. $h = \{v_1,v_3, v_4\}$. Let $h'$ be a hyperedge containing $wv_3$. Note that $h'$ must be free, otherwise it would either embeds $v_2 v_3$ or $v_3 v_4$, giving us a cycle with larger number of red $2$-edges. Now consider the cycle $$C' = v_1 v_4 v_5 v_6 v_7 w v_3 v_2 v_1$$ such that $v_1 v_4$ is embedded in $h$,  $v_7 w$ is embedded in $\phi(v_7 v_1)$,  $wv_3$ is embedded in $h'$, and every other $2$-edge of $C'$ is embedded the same way as in $C$.
    
    \item Case 3: $n=7$. By Claim \ref{cl:finite}, it suffices to consider the cycle in Figure \ref{fig:remaining}(c). Note that $v_4v_1$ is a bridge. Hence if the edge $h$ containing $v_4 v_1$ is free, then we are done by Claim \ref{cl:two-red-one-free}. Otherwise, WLOG, suppose $h$ is used to embed $v_3 v_4$, i.e., $h = \{v_1,v_3, v_4\}$. Moreover there are free edges $h_1$, $h_2$ (may be the same) such that $\{w,v_3\} \subseteq h_1$ and $\{w,v_6\} \subseteq h_2$. If $h_1 \neq h_2$, then consider the cycle
     \[v_1 v_4 v_5 v_6 w v_3 v_2 v_1\]
    such that $v_1 v_4$ is embedded in $h$, $v_6 w$ is embedded in $h_2$, $wv_3$ is embedded in $h_1$ and all other edges are embedded the same way as in $C$. We then obtain a Hamiltonian Berge cycle. On the other hand, suppose $h_1 = h_2$, then it follows that $h' = \{v_3, v_6,w\}$ is a free edge. Now consider the cycle 
    \[v_1 v_2 v_3 v_6 v_5 v_4 v_1\]
    such as $v_3v_6$ is embedded in $h'$, $v_4 v_1$ is embedded in $h$ and all other edges are embedded in the same way as before. Observe that this cycle, using the same coloring scheme as before, has three red edges, which contradicts our assumption that the cycle in Figure \ref{fig:remaining} has maximal number of red edges.

    \item Case 4: $n =6$. There are two possible coloring for $n = 6$ (see Figure \ref{fig:remaining2}$(a)(b)$). Let us first look at the case (Figure \ref{fig:remaining2}$(a)$) when there are two disjoint red segments of length $1$. Let $h_0$ be a hyperedge containing $w v_5$. Observe $h_0$ must be free, since otherwise it must be embedding $v_1 v_5$ or $v_4 v_5$, which contradicts that $v_1 v_5$ and $v_4 v_5$ are blue. Let $h_1, h_2$ be two hyperedges containing $v_1 v_3$ and $v_2 v_4$ respectively. 
    Note since $v_1 v_3$ and $v_2 v_4$ are bridges, if either of $h_1, h_2$ is free, then we are done by Claim \ref{cl:two-red-one-free}. Otherwise, there are two subcases:
        \begin{description}
            \item Case 4(a): $h_1 = \{v_1, v_3, v_5\}$ and $h_2 = \{v_2, v_4, v_5\}$. Let $h_3$ be a hyperedge containing $v_1 v_4$. It follows that $h_3$ must be free. Now consider the cycle
            $$w v_5 v_1 v_4 v_3 v_2 w$$
            such that $w v_5$ is embedded in $h_0$, $v_1 v_4$ is embedded in $h_3$, $wv_2$ is embedded in $\phi(v_1 v_2)$, and any other $2$-edge embedded the same way as in $C$. This is a Hamiltonian Berge cycle, which gives us a contradiction.
            
            \item Case 4(b): WLOG, $h_1 =\{v_1, v_2, v_3\}$ and $h_2 = \{v_2, v_4, v_5\}$. Then consider the cycle  
                $$w v_5 v_1 v_3 v_4 v_2 w$$
            such that $w v_5$ is embedded in $h_0$, $v_1 v_3$ is embedded in $h_1$, $v_4 v_2$ is embedded in $h_2$, $wv_2$ is embedded in $\phi(v_1 v_2)$, and any other $2$-edge embedded the same way as in $C$.
        \end{description}
    In both cases, we obtain a Hamiltonian Berge cycle. Hence we are done with the case in Figure \ref{fig:remaining}(a). The case in Figure \ref{fig:remaining}(b) is the same as Case 4(a).

\end{description}
\end{proof}

Given an $[k]$-graph $\cH=(V,E)$ and a subset $S\subseteq V$, the {\em trace} of $\cH$ on $S$ is defined to be the $[k]$-graph $\cH_S=(S, E')$ with the vertex set $S$ and the edge set $E':=\{F\cap S\colon F\in E(\cH)\}$. 
Traces of hypergraphs are very useful in extremal problems involving (non-uniform) hypergraphs. For some examples of results on trace functions and applications, see \cite{Sauer72, Shelah72, Vapnik-Chervonenkis71, Johnston-Lu14}.
Regarding the trace of covering hypergraphs, the following observations can be easily verified by definition.
\begin{proposition}\label{prop:cover}
Let $\cH$ be a $[k]$-graph and $S\subseteq V(\cH)$ be any subset of vertices. Then the following statements hold:
\begin{enumerate}
    \item If $\cH$ is covering, so is $\cH_S$.
    \item For each Berge-cycle (or Berge-path) in $\cH_S$, there is a Berge-cycle (or Berge-path) in $\cH$ with the same sequence of base vertices and thus with the same length.
\end{enumerate}
\end{proposition}

\begin{lemma}\label{lem:small3}
Let $\cH = (V,E)$ be a covering $[3]$-graph on $n\geq 4$ vertices. Then $\cH$ contains a Berge-triangle.
\end{lemma}
\begin{proof}
If all hyperedges have size $2$, then clearly we can find a Berge-triangle. Otherwise, take an arbitrary hyperedge $h = \{v_1, v_2, v_3\}$. Let $v_4$ be an arbitrary vertex not in $h$. Since $\cH$ is covering, there exists a hyperedge $h'$ containing $v_1 v_4$. Note that $h'$ contains at most one vertex from $\{v_2, v_3\}$. WLOG, $v_2 \notin h'$. Then there exists another hyperedge $h''$ containing $v_2 v_4$. It follows that we have a Berge-triangle $v_1 v_2 v_4$ such that $v_1 v_2$, $v_2 v_4$ and $v_4 v_1$ are embedded in $h, h'$ and $h''$ respectively.
\end{proof}

\begin{proof}[Proof of Theorem \ref{thm:3-cycle}]
Let $\cH$ be a covering $[3]$-graph on $n\geq 6$ vertices. We want to show that $\cH$ contains all Berge cycles of length $3\leq s\leq n$. 
Observe that given any $S\subseteq V(\cH)$ with $|S|\geq 6$, Lemma \ref{lem:3-Hamiltonian-cycle} implies that $\cH_S$ contains a Hamiltonian Berge cycle, which by Proposition \ref{prop:cover}, can be lifted to a Berge cycle of length $|S|$ in $\cH$. Hence $\cH$ contains Berge cycles of length $6 \leq s \leq n$.

\begin{claim}\label{cl:Berge-C5}
$\cH$ contains a Berge cycle of length $5$.
\end{claim}
\begin{proof}
We know that $\cH$ contains a Berge cycle $C$ of length $6$. Let $C = v_1 v_2 \cdots v_6$ be one of such Berge cycles. 
For convenience assume $v_{i}\equiv v_{\textrm{$i$ mod $6$}}$. Again call an hyperedge $h$ free if $h$ is not a hyperedge of the Berge cycle $C$. Now for each $i\in [6]$, if any hyperedge $h_i$ containing $v_i v_{i+2}$ is free or $h_i = \{v_i, v_{i+1}, v_{i+2}\}$, then we are done since we can obtain a Berge cycle $C' = v_1 \cdots v_i v_{i+2} \cdots v_6 v_1$ of length $5$ by embedding $v_i v_{i+2}$ in $h_i$ and every other $2$-edge the same way it is embedded in $C$.
Otherwise, for each $i \in [6]$, the hyperedge $h_i$ must be either $\{v_i ,v_{i+2}, v_{i+3}\}$ or $\{v_{i-1} v_{i} v_{i+2}\}$. 

\begin{description}
\item Case 1: there is some $i$ such that both $\{v_i ,v_{i+2}, v_{i+3}\}$ and $\{v_{i+1}, v_{i+3}, v_{i+4}\}$ (or both $\{v_i, v_{i+1}, v_{i+3}\}$ and $\{v_{i-1}, v_{i}, v_{i+2}\}$) are hyperedges of $C$. WLOG, by relabeling, assume that $\{v_1, v_3, v_4\}$ and $\{v_2, v_4, v_5\}$ are both in $C$. Then consider the cycle 
$$v_1 v_3 v_2 v_5 v_6 v_1$$
such as $v_1 v_3$ is embedded in $\{v_1, v_3, v_4\}$, $v_2 v_5$ is embedded in $\{v_2, v_4, v_5\}$ and every other $2$-edge is embedded the same way in $C$. We then obtain a Berge cycle of length $5$.

\item Case 2: 
WLOG, assume that the vertex pair $v_2 v_4$ is embedded in $\{v_1, v_2, v_4\}$. Since we are not in Case $1$, then $v_3 v_5$ must be embedded in $\{v_3, v_5, v_6\}$, $v_4 v_6$ must be embedded in $\{v_3, v_4, v_6\}$, etc. With this logic, 
we then obtain a hypergraph on $6$ vertices with at least the following hyperedges: $h_1 = \{v_1, v_2, v_4\}$, $h_2 = \{v_3, v_5, v_6\}$, $h_3 = \{v_3, v_4, v_6\}$, $h_4 = \{v_1, v_2, v_5\}$, $h_5 = \{v_2, v_5, v_6\}$, $h_6 = \{v_1, v_3, v_4\}$. Now consider the cycle 
$$v_2 v_5 v_6 v_3 v_4$$
by embedding the $2$-edges in $h_4, h_5, h_2, h_3, h_1$ respectively. We then have a Berge cycle of length $5$.
\end{description}
This completes the proof of the claim.
\end{proof}

\begin{claim}\label{cl:Berge-C4}
$\cH$ contains a Berge cycle of length $4$.
\end{claim}
\begin{proof}
The proof is very similar to that of Claim \ref{cl:Berge-C5}. By Claim \ref{cl:Berge-C5}, we know that $\cH$ contains a Berge cycle $C$ of length $5$. Let $C = v_1 v_2 \cdots v_5$ be one of such Berge cycles (and assume $v_{i}\equiv v_{\textrm{$i$ mod $5$}}$). Similar to before, if there is any an hyperedge $h_i$ containing $v_i v_{i+2}$ such that $h_i$ is free or $h_i = \{v_i, v_{i+1}, v_{i+2}\}$, then we can easily find a Berge cycle of length $4$. 

Now let $h_1$ be a hyperedge containing $v_1 v_3$. If $h_1$ is free or $h_1 = \{v_1 v_2 v_3\}$, then we are done. WLOG suppose that $h_1$ is used to embed $v_3 v_4$, i.e., $h_1 = \{v_1, v_3, v_4\}$. Let $h_2$ be a hyperedge containing $v_2 v_5$. Similarly, if $h_2$ is free or $h_2 = v_1 v_2 v_5$, then we are done. WLOG, suppose $h_2$ is used to embed $v_3 v_3$, i.e., $h_2 = \{v_2, v_3, v_5\}$. We then can obtain a Berge cycle of length $4$ by the same argument in Case 1 of Claim \ref{cl:Berge-C5}. This completes the proof of Claim \ref{cl:Berge-C4}.
\end{proof}

By Lemma \ref{lem:small3}, Claim \ref{cl:Berge-C5} and Claim \ref{cl:Berge-C4}, $\cH$ also contains Berge cycles of length $3$, $4$, and $5$. This completes the proof of Theorem \ref{thm:3-cycle}.
\end{proof}

\section{Proof of Theorem \ref{thm: max-Lagrangian-cycle}}

Before we show the proof of Theorem \ref{thm: max-Lagrangian-cycle}, we need a few definitions and lemmas.
For a vector $\bm{x} = (x_1, \ldots, x_n)$ of real numbers, the \textit{support} of $\bm{x}$ is defined as $Supp(\bm{x}):= \{1 \leq i\leq n: x_i \neq 0\}$. Given a hypergraph $\cF$ and $I \subseteq V(\cF)$, let $\cF[I]$ be the \textit{induced} subhypergraph of $\cF$, i.e., $V(\cF[I])= I$ and $E(\cF[I])= \{h\in \cF: h\subseteq I\}$.

\begin{lemma}[\cite{Frankl-Rodl84}]\label{lem:Lagrangian-cover}
Let $\cF$ be a $k$-uniform hypergraph on $n$ vertices. Suppose $\bm{x} =(x_1, \ldots, x_n)$ with $x_i\geq 0$ such that $\sum_{i=1}^n x_i = 1$. Moreover, suppose that $P_{\cF}( x) = \lambda(\cF)$ and $I = Supp(x)$ is minimal. Then $\cF[I]$ is covering.
\end{lemma}

\begin{proof}[Proof of Theorem \ref{thm: max-Lagrangian-cycle}]

Let $\cH$ be a Berge-$C_{t}$-free $k$-uniform hypergraph on $n$ vertices that achieves the maximum Lagrangian where $t \geq n_0(\{k\})$ in Theorem \ref{thm:cover-cycles}. Suppose that $\bm{x} =$ $(x_1, x_2, \ldots, x_n)$ $\in \mathbb{R}^n$ such that $x_i \geq 0$, $\sum_{i=1}^n x_i = 1$ and $P_{\cH}(\bm{x}) = \lambda(\cH)$. Further assume that $I = Supp(\bm{x})$ is minimal. By Lemma \ref{lem:Lagrangian-cover}, we have that $\cH[I]$ is covering. Since $\cH$ is Berge-$C_{t}$-free, it follows by Theorem \ref{thm:cover-cycles} that $\abs{I} \leq t-1$. Hence
\begin{align}
    \lambda(\cH) =  P_{\cH}(x) &= \dss_{\{i_1, i_2, \cdots, i_k\} \in E(\cH)} x_{i_1} x_{i_2} \cdots x_{i_k} \nonumber\\
                           & = \dss_{\substack{\{i_1, i_2, \cdots, i_k\} \in E(\cH)\\i_1, i_2, \cdots, i_k \in I}}x_{i_1} x_{i_2} \cdots x_{i_k} \nonumber \\
                           & \leq \dss_{\{i_1, i_2, \cdots, i_k\} \in \binom{I}{k}} x_{i_1} x_{i_2} \cdots x_{i_k}. \label{eq:lag}
\end{align}

Recall that $\dss_{i=1}^n x_i = 1$. Now we claim that \eqref{eq:lag} is maximized when $x_i = 1/n$ for all $i\in [n]$. Indeed, suppose otherwise that $x_s \neq x_t$ for some fixed $s\neq t$. Then consider a new vector $\bm{x'} =$ $(x'_1, x'_2, \ldots, x'_n)$ such that $x'_i = x_i$ for $i \in [n]\backslash \{s,t\}$ and $x'_s = x'_t =  \frac{1}{2}(x_s + x_t)$. It follows that 
\begin{align*}
      & \dss_{\{i_1, i_2, \cdots, i_k\} \in \binom{I}{k}} x'_{i_1}\cdots x'_{i_k} - \dss_{\{i_1, i_2, \cdots, i_k\} \in \binom{I}{k}} x_{i_1} \cdots x_{i_k} \\
     = &\dss_{i_1, i_2, \cdots i_{k-2} \in \binom{I\backslash \{s,t\}}{k-2}} x_{i_1} x_{i_2} \cdots x_{i_{k-2}} (x'_s x'_t - x_s x_t) \\
     = &\dss_{i_1, i_2, \cdots i_{k-2} \in \binom{I\backslash \{s,t\}}{k-2}} x_{i_1} x_{i_2} \cdots x_{i_{k-2}} \lp \frac{1}{4}(x_s+x_t)^2 - x_s x_t\rp\\
     > & 0.
\end{align*}
Hence it follows \eqref{eq:lag} is maximized when $x_i = 1/n$ for all $i\in [n]$. We then have 
\begin{align*}
     \lambda(\cH) =  P_{\cH}(x) 
                            & \leq \dss_{\{i_1, i_2, \cdots, i_k\} \in \binom{I}{k}} x_{i_1} x_{i_2} \cdots
                                    x_{i_k}\\
                           &\leq \frac{1}{(t-1)^{k}} \binom{t-1}{k}\\
                           &= \lambda(K^k_{t-1}).
\end{align*}
This completes the proof of Theorem \ref{thm: max-Lagrangian-cycle}.
\end{proof}

\begin{proof}[Proof of Corollary \ref{cor:3-cycle-Lagrangian} and \ref{cor:3-path-Lagrangian}]
For $k = 3$, Corollary \ref{cor:3-cycle-Lagrangian} and \ref{cor:3-path-Lagrangian} follows from Theorem \ref{thm:3-cycle} and Lemma \ref{lem:Lagrangian-cover} using the same logic as the proof of Theorem \ref{thm: max-Lagrangian-cycle}.
\end{proof}


\begin{thebibliography}{10}

\bibitem{AFR95}
M. Albert, A. Frieze, and B. Reed, Multicoloured Hamilton cycles, \textit{Electron. J. Comb.} \textbf{2} (1995), Research Paper R10, 13 p.


\bibitem{Alon-Spencer08}
N. Alon and J. H. Spencer, \textit{The Probabilistic Method}, Wiley-Interscience, third edition, 2008.

\bibitem{BGHS78}
J.-C. Bermond, A. Germa, M.-C. Heydemann, and D. Sotteau, Hypergraphes hamiltoniens, In \textit{Probl\`{e}mes combinatoires et th\'eorie des graphes (Colloq. Internat. CNRS, Univ. Orsay, Orsay, 1976)},
volume 260 of \textit{Colloq. Internat. CNRS}, pages 39--43. CNRS, Paris, 1978.

\bibitem{BKP12}
J. B\"{o}ttcher, Y. Kohayakawa, and A. Procacci. Properly coloured copies and rainbow copies of large graphs with small maximum degree. \textit{Random Structures \& Algorithms},  \textbf{40(4)}, (2012), 425–436.

\bibitem{Coulson-Perarnau18}
M. Coulson and G. Perarnau, A rainbow Dirac's theorem, \textit{arXiv}:1809.06392.

\bibitem{DGMT18}
A. Davoodi, E. Gy\H{o}ri, A. Methuku, and C. Tompkins, An Erd\H{o}s-Gallai type theorem for hypergraphs, \textit{European J. Combin.} \textbf{69} (2018), 159--162.


\bibitem{Dirac52}
G. A. Dirac, Some theorems on abstract graphs, \textit{Proc. London Math. Soc.}, \textbf{3} (1952), 69--81.

\bibitem{ENR83}
P. Erd\H{o}s, J. Nestril and V. R\"{o}dl, Some problems related to partitions of edges of a graph, in: \textit{Graphs and Other Combinatorial Topics} (Teubner, Leipzig, 1983) 54--63.

\bibitem{EGMSTZ19}
B. Ergemlidze, E. Gy\H{o}ri, A. Methuku, N. Salia, C. Tompkins, and O. Zamora, Avoiding long Berge cycles, the missing cases $k=r+1$ and $k=r+2$, \textit{Combin. Probab. Comput.}, (2009) 1--13. 

\bibitem{Frankl-Rodl84}
P. Frankl, V. R\"{o}dl, Hypergraphs do not jump, \textit{Combinatorica} \textbf{4} (1984) 149--159.


\bibitem{Frieze-Krivelevich08}
A. Frieze and M. Krivelevich, On rainbow trees and cycles, \textit{Electron J Comb.} \textbf{15} (2008), Research Paper R59, 9 p.

\bibitem{Frieze-Reed93}
A. Frieze  and  B. Reed, Polychromatic Hamilton cycles, \textit{Disc. Math.} \textbf{118}, (1993) 69--74.

\bibitem{FKL19}
Z. F\"{u}redi, A. Kostochka, and R. Luo, Avoiding long Berge cycles, \textit{J. Combin. Theory Ser. B}, \textbf{137} (2019), 55-64.

\bibitem{FKL19B}
Z. F\"{u}redi, A. Kostochka, and R. Luo, Avoiding long Berge cycles II, exact bounds for all $n$, arXiv:1807.06119.

\bibitem{Gerbner-Palmer17}
D. Gerbner, C. Palmer, Extremal results for Berge-hypergraphs, \textit{SIAM J. Discrete Math.} \textbf{31 (4)}
(2017) 2314–-2327.

\bibitem{GKL10} 
E. Gy\H{o}ri, G. Y. Katona and Nathan Lemons. Hypergraph extensions of the Erd\H{o}s-Gallai theorem. \textit{Electronic Notes in Discrete Mathematics}, \textbf{36} (2010), 655-662.

\bibitem{Gyori-Lemons12}
E. Gy\H{o}ri, N. Lemons, 3-uniform hypergraphs avoiding a given odd cycle, \textit{Combinatorica} \textbf{32 (2)} (2012) 187-–203.

\bibitem{Gyori-Lemons12b}
E. Gy\H{o}ri, N. Lemons, Hypergraphs with no cycle of a given length, \textit{Combin. Probab. Comput.} \textbf{21 (1–2)} (2012) 193-–201.


\bibitem{GLSZ18}
E. Gy\H{o}ri, N. Lemons, N. Salia, O. Zamora, The structure of hypergraphs without long Berge cycles, arXiv:1812.10737. 

\bibitem{Hahn80}
G. Hahn, Un jeu de colouration, In \textit{Actes du Colloque de Cerisy}, volume 12, pages 18--18, 1980.

\bibitem{Hahn-Thomassen86}
G. Hahn and C. Thomassen, Path and cycle sub-Ramsey numbers and an edge-colouring conjecture, \textit{Disc. Math.} \textbf{62} (1986), 29--33.


\bibitem{Han-Zhao15}
J. Han, Y. Zhao, Minimum codegree threshold for Hamilton $\ell$-cycles in $k$-uniform hypergraphs, \textit{J. Combin. Theory Ser. A} \textbf{132(0)} (2015) 194–-223.

\bibitem{Johnston-Lu14}
T. Johnston, L. Lu, Tur\'an problems on non-uniform hypergraphs, \textit{The Electronic Journal of Combinatorics} \textbf{21(4)}  (2014) 4-–22.



\bibitem{Katona-Kierstead99}
 G. Katona and H. Kierstead, Hamiltonian chains in hypergraphs, \textit{Journal of Graph Theory}, \textbf{30(2)} (1999), 205--212.
 
\bibitem{Keevash11} P. Keevash, Hypergraph Tur\'an problems, {\it Surveys in Combinatorics}, Cambridge University Press, 2011, 83-–139.

\bibitem{Kostochka-Luo18}
A. Kostochka, R. Luo, On $r$-uniform hypergraphs with circumference less than $r$, arXiv:1807.04683.

\bibitem{KLZ19}
A. Kostochka, R. Luo and D. Zirlin, Super-pancyclic hypergraphs and bipartite graphs, arXiv:1905.03758.

\bibitem{KMO10}
D. K\"{u}hn, R. Mycroft, and D. Osthus, Hamilton $\ell$-cycles in uniform hypergraphs, \textit{Journal of Combinatorial Theory. Series A}, \textbf{117(7)} (2010), 910–-927.  

\bibitem{Lazebnik-Verstraete03}
F. Lazebnik, J. Verstra\"{e}te, On hypergraphs of girth five, \textit{Electron. J. of Combin.},
\textbf{10}, (2003), \#R25.


\bibitem{Maamoun-Meyniel84}
M. Maamoun and H. Meyniel, On a problem of G. Hahn about coloured Hamiltonian paths in $K_{2t}$,
\textit{Disc. Math.}, \textbf{51(2)} (1984), 213--214.

\bibitem{Markstrom-Rucinski11}
K. Markstr\"{o}m and A. Ruci\'nski, Perfect matchings (and Hamilton cycles) in hypergraphs with large degrees, \textit{Eur. J. Comb.}, \textbf{32(5)} (2011), 677–-687.

\bibitem{Motzkin-Straus65}
T. S. Motzkin and E. G. Straus, Maxima for graphs and a new proof of a theorem of Tur\'{a}n, {\it Canad. J. Math.} {\bf17}  (1965) 533--540.
 
\bibitem{RRE06}
V. R\"odl, A. Ruci\'nski, and E. Szemer\'edi, A Dirac-type theorem for 3-uniform hypergraphs, \textit{Combinatorics,
Probability and Computing}, \textbf{15(1-2)} (2006), 229--251.
 
\bibitem{RRE08}
V. R\"odl, A. Ruci\'nski, and E. Szemer\'edi, An approximate Dirac-type theorem for $k$-uniform hypergraphs,
\textit{Combinatorica}, \textbf{28(2)} (2008), 229--260.
 
\bibitem{RRE11}
V. R\"odl, A. Ruci\'nski, and E. Szemer\'edi, Dirac-type conditions for Hamiltonian paths and cycles in
3-uniform hypergraphs, \textit{Advances in Mathematics}, \textbf{227(3)} (2011), 1225--1299.
 
\bibitem{Sauer72}
N. Sauer, On the density of families of sets, \textit{J. Combin. Theory Ser. A} \textbf{13} (1972), 145--147. 
 

\bibitem{Shelah72} 
S. Shelah, A combinatorial problem; stability and order for models and theories in infinitary
languages, \textit{Pacific or. Math.} \textbf{41} (1972), 271--276.

 \bibitem{Talbot02}
 J. Talbot, Lagrangians of hypergraphs, {\it Combin., Probab. Comput.} {\bf 11} (2002) 199--216.
 
\bibitem{Treglown-Zhao12}
A. Treglown and Y. Zhao, Exact minimum degree thresholds for perfect matchings in uniform hypergraphs, \textit{J. Combin. Theory Ser. A}, \textbf{119(7)} (2012) 1500–-1522. 
 
\bibitem{Vapnik-Chervonenkis71}
V. N. Vapnik and A. Ya. Chervonenkis, On the uniform convergence of relative frequencies of
events to their probabilities, \textit{Theory Probab. Appl.} \textbf{16} (1971), 264--280. 
 
 
\bibitem{Zhao16}
Y. Zhao, \textit{Recent advances on Dirac-type problems for hypergraphs}, Recent Trends in Combinatorics (Andrew Beveridge, R. Jerrold Griggs, Leslie Hogben, Gregg Musiker, and Prasad Tetali, eds.), Springer International Publishing, 2016, pp. 145--165.
 
\end{thebibliography}
\end{document}